\documentclass[a4paper,11pt]{article}

\usepackage[a4paper]{geometry}
\usepackage[OT2,T1]{fontenc}

\usepackage{lipsum} 

\usepackage[english,francais]{babel}
\DeclareSymbolFont{cyrletters}{OT2}{wncyr}{m}{n}
\DeclareMathSymbol{\Sha}{\mathalpha}{cyrletters}{"58}

\usepackage{amssymb,amsmath,amsthm,amscd}
\usepackage{mathrsfs}

\usepackage{subfig}

\usepackage{tabularx} 
\usepackage{calc} 
\usepackage{graphicx} 

\usepackage{hyperref}

\usepackage[nottoc]{tocbibind}

\usepackage{makeidx}
\makeindex

\usepackage{longtable}
\usepackage{enumerate}

\usepackage{amsthm}
\usepackage{amssymb}
\usepackage{amsmath}
\usepackage{array}
\usepackage{mathrsfs}
\usepackage[utf8]{inputenc}
\usepackage[T1]{fontenc}
\usepackage{version}
\usepackage{lscape}
\usepackage{epigraph}
\usepackage[all]{xy}
\usepackage{datetime}
\usepackage{todonotes}
\usepackage{graphicx}
\usepackage{multirow}

\theoremstyle{definition}
\newtheorem{definition}{Definition}
\newtheorem{ex}{Example}
\theoremstyle{plain}

\newtheorem{thm}{Theorem}[section]
\newtheorem{lm}[thm]{Lemma}
\newtheorem{propo}[thm]{Proposition}
\newtheorem{coro}[thm]{Corollary}

\theoremstyle{remark}
\newtheorem{rem}{Remark}

\def\Z{\mathbb{Z}}
\def\N{\mathbb{N}}
\def\Q{\mathbb{Q}}
\def\R{\mathbb{R}}
\def\E{\mathscr{E}}
\def\P{\mathbb{P}}

\def\F{\mathscr{F}}

\title{Root number of twists of an elliptic curve}
\author{Julie DESJARDINS}
\date{}

\begin{document}

\begin{otherlanguage}{english}

\maketitle

\abstract{
We give an explicit description of the behaviour of the root number in the family given by the twists of an elliptic curve $E/\Q$ by the rational values of a polynomial $f(T)$. In particular, we present a criterion for the family to have a constant root number over $\Q$. This completes work by Rohrlich: we detail the behaviour of the root number when $E$ has bad reduction over $\Q^{ab}$ and we treat the cases $j(E)=0,1728$ which were not considered previously.
}


\section{Introduction}

This paper is concerned with the behaviour of the root number in a one-parameter family of twists of an elliptic curve by the values of a polynomial $f\in\Z[T]$, or equivalently, in the fibres of an isotrivial elliptic surface. We will see that the root number respects a certain type of periodicity, and we will give a criterion to predict when it is constant.

Let $E$ be an elliptic curve over $\Q$. The \textbf{root number} $W(E)$ is defined as the product of the local root numbers $W_p(E)\in\{\pm1\}$: $$W(E)=\prod_{p\leq\infty}{W_p(E)},$$
where $p$ runs through the finite and infinite places of $\Q$. These local factors, defined in terms of the epsilon factors of the Weil-Deligne representation of $\Q_p$ and explained in details by Rohrlich in \cite{Rohr2}, have the property that $W_p(E)=1$ for all but finitely many $p$. Rohrlich \cite{Rohr} gives an explicit formula for the local root numbers in terms of the reduction of the elliptic curve $E$ at a prime $p\not=2,3$. Moreover, we always have $W_\infty(E)=-1$. The remaining cases $p=2,3$ are covered by Halberstadt \cite{Halb} (see also Rizzo \cite{Rizz}). 

Let $L(E,s)$ denote the $L$-function of $E$. Then $W(E)$ is equal to the sign of the functional equation of $L$ by the Modularity Theorem over $\Q$. Note that over general number fields $K\not=\Q$, such an equality is only conjectural.

The Birch and Swinnerton-Dyer conjecture implies the following statement, known as the parity conjecture $$W(E)=(-1)^{\mathrm{rank }E(\Q)}.$$

A consequence of this conjectural equality is that it suffices to have $W(E)=-1$ for the rank of $E(\Q)$ to be non-zero, and in particular for $E(\Q)$ to be infinite.

By a \textbf{one-parameter family of elliptic curves} we mean the collection of the fibres of an elliptic surface over $\Q$ (with a section) given by a Weierstrass equation $$\E:y^2=x^3+A(T)x+B(T),$$
where $A(T),B(T)$ are polynomials with coefficients in $\Z$, and the discriminant is denoted by $ \Delta(T)$. For $t\in\P^1$ such that $\Delta(t)\not=0$, the \textbf{fibre at $t$}, $\E_t:y^2=x^3+A(t)x+B(t)$, is an elliptic curve. In this paper we consider the case where the family is isotrivial, i.e. when its $j$-invariant function $t\mapsto j(\E_t)=\frac{4A(t)^3}{\Delta(t)}$ is constant. In that case, the curves $\E_t$ are twists of one another, and $\E$ can be seen as a subfamily of the families of all twists of $\E_1$:
\begin{enumerate}
\item (quadratic twists, $j(\E)\not=0,1728$) $y^2=x^3+aH(T)^2x+bH(T)^3$,
\item (quartic twists, $j(\E)=1728$) $y^2=x^3+A(T)x$,
\item (sextic twists, $j(\E)=0$) $y^2=x^3+B(T)$,
\end{enumerate}
where $a,b\not=0$ and $H(T)$, $A(T)$, $B(T)$ are non-zero polynomials with integer coefficients.

In a previous article \cite{Desjardins1}, the author proves that the function $t\mapsto W(\E_t)$ defined by the root number on a non-isotrivial family $\E$ is never periodic (i.e. constant on a congruence class of $t$). More precisely, she proves that the sets $W_{\pm}=\{t\in\P^1 \ \vert \ W(\E_t)=\pm1\}$ are both infinite, which implies (under the parity conjecture) the Zariski-density of the rational points $\E(\Q)$.

For families of twists thought, it can happen that $W(\E_t)$ takes the same value for every $t\in\P^1$ associated to a smooth $\E_t$, and (more alarmingly if one is interested in proving the Zariski-density) equal to $+1$, as observed previously in each of the two special cases:
\begin{enumerate}
\item (Cassels and Schinzel \cite{CS}) $y^2=x^3-(1+T^4)^2x$,
\item (V\'arilly-Alvarado \cite{VA}) $y^2=x^3+27T^6+16$.
\end{enumerate}
Observe that the first example is a K3 surface and that the second is a rational elliptic surface. These specific surfaces have however a Zariski-dense set of rational points - the proof can be found respectively in \cite{Zhizhong} and in \cite[Example 7.1]{VA}. By the time the present article was published, the author and B. Naskr\c ecki \cite{DesjardinsBartosz} released a preprint presenting a simple algorithm to find the generic rank of any elliptic surface of the form $y^2=x^3+AT^6+B$, which generalizes V\'arilly-Alvarado's result.

It is not possible that all quadratic twists have the same root number, as observed by Dokchitser and Dokchitser \cite{Dokchitserquadratictwists}. However, they proved that the elliptic curve \begin{equation}\label{surface3}y^2=x^3+x^2-12x-\frac{67}{4}\end{equation} has root number $+1$ over every extension of $K=\Q(\sqrt[4]{-37})$.
This curve has the additional property that any twist by an integer $t\in\N$ has root number $+1$ over $\Q$ and any twist by $-t$ has root number $-1$. Thus, polynomials $f$ with only positive values (resp. only negative values) define a family of twists with constant root number $+1$ (resp. $-1$). For instance, $(T^2+1)y^2=x^3+x^2-12x-\frac{67}{4}$ has constant root number $+1$. However, the density of the rational points is proven on every family of quadratic twists of an elliptic curve by a $f$ with degree $\leq 2$ \cite[Theorem 4.2.]{Desjardins2}.

In this article, we describe for which elliptic curves $E$ the root number of a twist by $t\in\Z$ only depends on the sign of $t$ (Lemma \ref{brot}). Moreover, for general elliptic curves $E$ we describe the behaviour of the root number (Theorem \ref{introiso}).

\subsection{Notation}

Throughout the paper we use two non-standard notations. Given an integer $\alpha\in\Z$ and a prime number, we denote by $\alpha_{(p)}$ the integer such that $$\alpha=p^{v_p(\alpha)}\alpha_{(p)}.$$  Similarly, we define $\alpha_{(d)}$ with $d=\prod{p_i^{e_i}}\in\N$ as the integer $$\alpha=\left(\prod_{i} p_i^{v_{p_i}(\alpha)}\right)\alpha_{(d)}.$$
We will denote by $sq(\alpha)$ and call the \textit{square part} of $\alpha$ the integer $$sq(\alpha)=\prod_{i\text{ such that }e_i\text{ is even.}}{p_i}$$
Incidentally, $sq(\alpha)_{(p)}$ refers to the integer $$sq(\alpha)_{(p)}=\prod_{i}{p_i},$$
where i ranges on $e_i$ even and $p_i\not=p$. We also call $\mathrm{sgn}(\alpha)$ the sign of $\alpha$.
\subsection{Main results}

We are interested in the behaviour of the root number in a one-parameter family of twists of elliptic curves, that is, to say in the fibres of an isotrivial elliptic surface. 
\begin{definition}
\begin{enumerate}
\item A function $f:\mathscr{F}\subseteq\Z\rightarrow\{\pm1\}$ is periodic (or $N$-periodic) if there exists a positive integer $N$ such that for each $t,t'\in\mathscr{F}$
$$t\equiv t'\mod N \Rightarrow f(t)=f(t').$$

We denote a \emph{congruence class modulo $N$} by $[t]$ (where $t\in\Z$ is a representative of the congruence class).

\item A function $f:\mathscr{F}\subseteq\Z\rightarrow\{\pm1\}$ is square-periodic (or $(N,M)$-square-periodic) if there exist positive integers $N, M$ such that for $t,t'\in\mathscr{F}$ that we write in its factorisation into prime factors $t=p_1^{e_1}\cdot p_r^{e_r}$ and $t'=q_1^{e'_1}\cdot q_s^{e'_s}$, we have
$$t\equiv t'\mod N\text{ and }sq(t)\equiv sq(t')\mod M \Rightarrow f(t)=f(t'),$$
where $sq(t)=\prod_{e_i \text{ even}}{p_i}$ and $sq(t')=\prod_{e_i' \text{ even}}{q_i}$.

We call a \emph{square-congruence class} $[t]_{sq}\mod N,M$ the set of the integers $t'$ such that $t'\equiv t\mod N$ and $sq(t')\equiv sq(t)\mod M$.

\end{enumerate}
\end{definition}


Let $E$ be the elliptic curve defined by the Weiestrass equation $y^2=x^3+ax+b$ and let $t\in\Q\setminus\{0\}$. Then the \emph{twist by $t$ of $E$} is the elliptic curve $E_t$ given by the following Weiestrass equation:
\begin{enumerate}
\item $E_t:y^2=x^3+at^2x+bt^3\quad$ if $j(E)\not=0,1728$ (i.e. $ab\not=0$),
\item $E_t:y^2=x^3+atx\quad$ if $j(E)=1728$ (i.e. $b=0$),
\item $E_t:y^2=x^3+bt\quad$ if $j(E)=0$ (i.e. $a=0$).
\end{enumerate}

Note also the following isomorphisms for all $t\in\Q\setminus\{0\}$: $E_{t^2}\cong E$ if $j(E)\not=0,1728$,
 $E_{t^4}\cong E$ if $j(E)=1728$,
 $E_{t^6}\cong E$ if $j(E)=0$.
As a consequence, for any $ t=\frac{p}{q}\in\Q\setminus\{0\}$, the twist $E_t$ is isomorphic to $E_{t'}$ for 
 $t'=qp$ if $ab\not=0$;
 $t'=pq^3$ if $b=0$;
 $t'=pq^5$ if $a=0$. It is thus sufficient to study the twists by integers.

The results of this article are summarized in the following theorem.

\begin{thm}\label{introiso}
Let $E$ be an elliptic curve and for $t\in\Z$ denote by $E_t$ its twist by $t$.

\begin{enumerate}
\item Suppose that $j(E)\not=0,1728$. Define $\mathscr{F}_2$ to be the set of squarefree integers, and $\mathscr{F}_2^+$ (respectively $\mathscr{F}_2^-$) the subset of $t\in\mathscr{F}_2$ with $\mathrm{sgn}(t)=+1$ (resp. $\mathrm{sgn}(t)=-1$).
Then

\begin{enumerate}
\item The root number can be written as the following product
$$W(E_t)=-W_2(E_t)W_3(E_t)\left(\frac{-1}{\vert t_{(6\Delta)}\vert}\right)\left(\prod_{p\mid\Delta_{(6)}}{W_p(E_t)}\right)
$$
where 
 $\left(\frac{\cdot}{\cdot}\right)$ is the Jacobi symbol.
 
\item\label{intro1b} the function $t\mapsto W(E_t)$ is periodic on $\mathscr{F}_2^\pm$.
\item The root number $W(E_t)$ is not constant when $t$ runs through $\mathbb{Z}\setminus\{0\}$. However, if $E$ satisfies the properties of Lemma \ref{brot}, it is constant on $\mathbb{Z}_{<0}$ and $\mathbb{Z}_{>0}$.
\end{enumerate}

\item Suppose that $j(E)=1728$. Define $\mathscr{F}_4$ to be the set of fourth-powerfree integers, and $\mathscr{F}_4^+$ (respectively $\mathscr{F}_4^-$) the subset of $t\in\mathscr{F}_4$ with $\mathrm{sgn}(t)=+1$ (resp. $\mathrm{sgn}(t)=-1$). Then 
\begin{enumerate}
\item \label{introsigneisoTx} The root number can be written as the following product
\begin{displaymath}
W(E_t)=-W_2(E_t)W_3(E_t)\left(\frac{-2}{\vert t_{(6)}\vert}\right)\left(\frac{-1}{sq(t)_{(6)}}\right).
\end{displaymath}

\item The function $t\mapsto W(E_t)$ is square-periodic on $\mathscr{F}_4^\pm$.

\item The root number $W(E_t)$ is not constant when $t$ runs through $\Z\setminus\{0\}$.
\end{enumerate}

\item Suppose that $j(E)=0$. Define $\mathscr{F}_6$ to be the set of sixth-powerfree integers, and $\mathscr{F}_6^+$ (respectively $\mathscr{F}_6^-$) the subset of $t\in\mathscr{F}_6$ with $\mathrm{sgn}(t)=+1$ (resp. $\mathrm{sgn}(t)=-1$). Then

\begin{enumerate}

\item \label{introsigneisoT} The root number can be written as the following product
\begin{displaymath}
W(E_t)=-W_2(E_t)W_3(E_t)\left(\frac{-1}{\vert t_{(6)}\vert}\right)\left(\frac{sq(t)_{(6)}}{3}\right).
\end{displaymath}

\item The function $t\mapsto W(E_t)$ is square-periodic on $\mathscr{F}_6^\pm$,

\item The root number $W(E_t)$ is not constant when $t$ runs through $\Z\setminus\{0\}$. 
\end{enumerate}

\end{enumerate}

\end{thm}




The article is organised as follows. In the rest of the introduction, we relate our results to previous work, in particular to Rohrlich's. In Section \ref{monodromiefacile} we study the monodromy of the reduction on a family of twists in each of the three cases $j\in\Q\setminus\{0,1728\}$, $j=1728$ and $j=0$, and in Section \ref{localrootnumber} we use it to describe the variation of the root number between an elliptic curve and one of its twists. We conclude the paper in Section \ref{proof} by proving Theorem \ref{introiso} in each of the three cases.


\subsection{Recollection on Rohrlich's results}

Theorem \ref{introiso} completes the following result on the variation of the root number of quadratic twists due to Rohrlich, and extends it to the case $ab=0$: the quartic and sextic twists families. 

\begin{thm} \cite[Theorem 2]{Rohr}
Let $a,b\in\Z\setminus\{0\}$ such that $4a^3+27b^2\not=0$. Consider the elliptic curve given by the equation $E:y^2=x^3+ax+b$. Let $f(t)\in\Z[t]$ and the family of quadratic twists given by the equation

$$E_{f(t)}:y^2=x^3+af(t)^2x+bf(t)^3.$$

Then, one of the two properties holds :
\begin{enumerate}
\item The sets $W_+$ and $W_-$ are dense in $\R$.
\item The sets $W_+$ and $W_-$ are $\{t\in\Q\ \vert \ f(t)<0\}$ and $\{t\in\Q\ \vert \ f(t)>0\}$ (in either order).
\end{enumerate}
Moreover, given $E$:
\begin{enumerate}[a.]
\item there exists $f$ such that the second assertion holds and such that the number of change of the sign of $f$ over $\R$ is greater that any given value.
\item if $E$ has good reduction on an abelian extension of $\Q$, then the second assertion holds. 
\end{enumerate}
\end{thm}

As a conclusion of this theorem, one obtains :

\noindent 1) if $E$ has good reduction over $\Q^{ab}$ then $E_{f(t)}$ has constant root number if and only if $f(t)$ takes the same sign for all $t\in\Q$,

\noindent 2) if $E$ has bad reduction over $\Q^{ab}$, then

\noindent a. if $f(t)$ does not always take the same sign, then $W_+$ and $W_-$ have infinite cardinality.
	
\noindent b. if $f(t)>0$ (or $<0$), then Rohrlich's theorem does not allow us to conclude directly. In order to know if the root number of the fibres is constant or not, the use of Theorem \ref{introiso} is necessary. Here is how to proceed:

\begin{enumerate}[(1)]
\item Find $N$ the smallest integer such that for each $t,t'\in\mathscr{F}_2$ the congruence $t\equiv t'\mod N$ implies $W(\E_t)=W(\E_{t'})$, or in other terms the smallest $N$ for which the root number function $t\longrightarrow W(E_t)$ is $N$-periodic. The existence of this $N$ is given by Theorem \ref{introiso}. The value of $N$ depends on the coefficients of a Weierstrass equation of $E$ and can be found with the help of Corollary \ref{decompositiontwist}). 
 
\item Then determine in which of the equivalence classes modulo $N$ are the values of the squarefree factors of $f(t)$. Take representatives $t_1,...,t_n$ such that each of the $f(t_i)$ represents a class modulo $N$ (that can be obtained). If the root number of the fibres of $E_{f(t_i)}$ all have the same value, then the root number function is constant and we are in Rohrlich's case 2. Otherwise, it varies and we are in case 1. 
\end{enumerate}

Let us explain this with an example:

\begin{ex}\label{ex1}
Let $a=2\cdot7\cdot17=238$ and $b=2^3\cdot7\cdot17=952$, and $f(t)=t^{86}+14$. We study the variation of the root number in the fibres of the family $$E_{f(t)}:y^2=x^3+238(t^{86}+14)^2x+952(t^{86}+14)^3.$$ Note that $E:y^2=x^3+238x+952$ has multiplicative reduction over $\Q^{ab}$ (since $173$ is a place of multiplicative reduction over $\Q$). Indeed, we have $\Delta(E)=-2^9\cdot7^2\cdot17^2\cdot173$.

First find the integer $N$ for which $E$ is $N$-periodic. For this, observe that the bad places are $2$, $7$ (type $II$), $17$ (type II) and $173$ (type $I_1$). Then by Corollary \ref{decompositiontwist} the root number of $E_{t}$ the twist by $t$ is 
$$W(E_{t})=-\prod_{p=2,3,7,17,173}{w_p(t)},$$
where $w_p(t)$ are local contributions at each $p$ determined as follows:
$$w_p(t)=\begin{cases}\mathrm{sgn}(t)D_2\left(\frac{-1}{\mid t_{(2)}\mid}\right)W_2(E)& \text{for }p=2\\ (-1)^{v_3(t)}D_3W_3(E)&\text{for }p=3\\ D_p\left(\frac{-1}{p}\right)^{v_p(t)}W_p(E)&\text{for }p=7,17\\ 
\left(\frac{t_{(173)}}{173}\right)D_pW_p(E)&\text{for $p=173$,}
\end{cases}$$
and the $D_p\in\{\pm1\}$ are the values depending on $t$ such that $W_p(E_t)=D_pW_p(E)$. They are given by the Lemmas \ref{signe2special}, \ref{signe3special} and \ref{variationquadratique}. Although an appropriate $D_p$ is not always given by Lemma \ref{signe2special} or \ref{signe3special} for $p=2,3$ in the case of a general elliptic curve, as is the case in this example. Example \ref{ex3} explains what to do otherwise.
\begin{itemize}
\item For 2: we have $(v_2(a)+4,v_2(b)+5,v_2(\Delta))=(5,8,9)$, $b_{(2)}\equiv3\mod4$ and $a_{(2)}\equiv7\mod8$. According to Lemma \ref{signe2special}, $w_2(t)$ the local contribution at 2 of the root number of the twists by $t$ takes the same value for any value of $t_{(2)}$ and of $v_2(t)$: $w_2(t)=+1$.
\item For 3: we have $(v_3(a)+1,v_3(b)+3,v_3(\Delta))=(1,3,0)$. According to Lemma \ref{signe2special}, $w_3$ the local contribution at 3 of the root number of the twists by $t$ will take the same value for any value of $t_{(3)}$ and of $v_3(t)$: $w_3(t)=+1$.
\end{itemize}
For primes $p\not=2,3$, Proposition \ref{variationquadratique} gives the value of $D_p$.
\begin{itemize}

\item For $7$ and $17$, the type of reduction is $II$ and we have: $$w_p(t)=\begin{cases}\left(\frac{3}{p}\right)\left(\frac{-1}{p}\right)W_p(E)&\text{if }p\mid t\\W_p(E)& \text{if $p\nmid t$.}\end{cases}$$ 
Since $7\equiv 1\mod6$ we have $w_7(t)=W_7(E)=-1$ for all $t$ squarefree integer, and since $17\not\equiv1\mod6$ we have $$w_{17}(t)=\begin{cases}W_{17}(E)=+1& \text{ if $v_{17}(t)$ even}\\
-W_{17}(E)=-1&\text{ if $v_{17}(t)$ odd.}\end{cases}$$

\item For $173$: the reduction is multiplicative, we have $$w_{173}(E_t)=\begin{cases}\left(\frac{t}{173}\right)W_{173}(E)=\left(\frac{t}{173}\right)&\text{if }173\nmid t\\ -\left(\frac{-6b_{(173)}}{173}\right)\left(\frac{t_{(173)}}{173}\right)W_{173}(E)=\left(\frac{t_{(173)}}{173}\right)
& \text{if }173\mid t\end{cases}$$
so we have simply $w_{173}(E_t)=\left(\frac{t_{(173)}}{173}\right)$.
\end{itemize}
As a consequence, we see that $W(E_t)$ takes the same value on the congruence classes modulo $17^2\cdot173$, so that the root number of the twists is $N$-periodic for $N=17^2\cdot173$. It is however somehow more convenient in this case to observe the more precise fact that it takes the same value for $t$ and $t'$ such that $t_{(173)}\equiv t'_{(173)}\mod 173$ and with $v_{17}(t)\equiv v_{17}(t')\mod 2$. 

An easy consequence of Fermat's little theorem is that $f(t)=t^{86}+14$ takes the values among $\{13,14,15\} \mod 173$. Thus we have $w_{173}(f(t))=\left(\frac{f(t)_{(173)}}{173}\right)$. This Jacobi symbol is equal to $+1$ for all squarefree $t$ (since $\left(\frac{11}{173}\right)=\left(\frac{12}{173}\right)=\left(\frac{13}{173}\right)=+1$). 

It is not so hard to check that $f(t)$ takes values among $\{1,5,6,10,12,13,15,16\} \mod 17$, and in particular that $17\nmid f(t)$ for any value of $t$. Thus $w_{17}(f(t))=+1$.  

This proves that the root number is constant on the family $E_{f(t)}$ and always takes the value $+1$.
\end{ex}

Sometimes, the computation of $N$ is not even necessary since a basic check proves that the root number varies:

\begin{ex}
If we twist $E:y^2=x^3+238x+952$ by the polynomial $g(t)=t^{86}+1$ instead: we have $W(E_{g(0)})=W(E)=1$ and $W(E_{g(1)})=W(E_2)=-1$
so the root number is not constant on the family $E_{g(t)}$. (When we look in detail, we see that the variation comes from $w_{173}$: we have $w_{173}(1)=+1$ and $w_{173}(2)=-1$. The other contributions are such that $w_p(1)=w_p(2)$.)
\end{ex}

In some other cases, it is more convenient to take a shortcut when we search for $N$, in particular when the local root number at $2$ or $3$ is not listed in Lemma \ref{signe2special} or \ref{signe3special}. Here is an example:

\begin{ex}\label{ex3}
In our first examples, finding the appropriate $N$ was easy because the functions $w_2(t)$ and $w_3(t)$ were constant by Lemma \ref{signe2special} and \ref{signe3special}. Let us choose another base curve, say $E':y^2=x^3+2\cdot17x+2^2\cdot17$, and study the family of twists of $E'$ by the values of the function $h(t)=8t^{30}+5$:
$$E'_{h(t)}:y^2=x^3+34(8t^{30}+5)^2x+68(8t^{30}+5)^3.$$ 
The discriminant of $E'$ is $\Delta(E')=-2^817^261$, so the reduction of $E'$ at 17 has type $II$ and the reduction at 61 is $I_1$. The root number can be written as 
$$W(E'_{t})=-\prod_{p=2,3,17,61}{w_p(t)},$$
with $w_p(t)$ as in Corollary \ref{decompositiontwist}. Let us find them explicitly in this case.
\begin{itemize}
\item  For 2:  we have $(v_2(a)+4,v_2(b)+5,v_2(\Delta))=(5,7,8)$, and this triple is not listed in Lemma \ref{signe2special} as one with $w_2(t)$ constant on the family $E'_t$. Some additional work must be done here with Rizzo's table III \cite{Rizz}. From there we extract the formula (for odd $t$):
$$W_2(E'_t)=\begin{cases}
+1 & \text{if }6\cdot17t^2+27\cdot17t^3\equiv1\mod8\text{ or }27\cdot17t^3\equiv5\mod 8\\ 
-1 & \text{otherwise}.
\end{cases}$$
In our choice of $h(t)=8t^{30}+5$, not only is $h(t)$ always positive and odd, but it is also always such that $h(t)\equiv5\mod8$. 
Consequently, $102h(t)^2+459h(t)^3\equiv59925\equiv5\mod8$ and $459h(t)^3\equiv7\mod8$ for any $t$. This means that $W_2(E'_{h(t)})=W_2(E')=-1$ (in particular $D_2=+1$ for all $t$), and thus that $w_2(h(t))=\mathrm{sgn}(h(t))\left(\frac{-1}{h(t)}\right)D_2W_2(E')=-1$ for all $t\in\Z$.
\item By Lemma \ref{signe3special}, $w_3(t)=+1$ for all $t$. 
\item Similarly as in Example \ref{ex1}, we have
$$w_{17}(t)=\begin{cases}+1&\text{if $v_{17}(t)$ even},\\-1&\text{if }v_{17}(t)\text{ odd}\end{cases}$$ and $$w_{61}(t)=\begin{cases}-\left(\frac{t}{61}\right)&\text{if }61\nmid t\\\left(\frac{t}{61}\right)&\text{if }61\mid t.\end{cases}$$
\end{itemize}
As a consequence, the root number $W(E'_{h(t)})$ is $17^2\cdot61$-periodic. We can be even more precise.
A consequence of Fermat's little theorem is that $h(t)=8t^{30}+5$ takes the values among $\{5,13,58\} \mod 61$. We have thus $w_{61}(h(t))=-\left(\frac{h(t)}{61}\right)=-1$ for all $t$. The polynomial $h(t)$ takes values among $\{1,3,4,5,6,7,9,13,14\}$ modulo $17$ and hence we always have $w_{17}(h(t))=+1$ for all $t$. Thus we have for every $t\in\Z$:
$$W(E'_{h(t)})=-(-1)(+1)(+1)(-1)=-1.$$
\end{ex}

\subsection{More formulae}

Birch and Stephens \cite{BS} prove formulae for the root number of $y^2=x^3-Dx$, and for the root number of $z^3=x^3+A$ (this curve can be rewritten as the equation $y^2=x^3-432A^2$). Liverance \cite{Liv} completes these results by giving a formula for the root number of $y^2=x^3+D$ in the general case.

The formulae given in the points \ref{introsigneisoT} and \ref{introsigneisoTx} of Theorem \ref{introiso} have a flavor different from that found in those two papers, in particular, it distinguishes between primes $p\geq5$ according to whether or not $p^2\mid t$, in a similar way to the formulae of V\'arilly-Alvarado \cite[Prop. 4.4 and 4.8]{VA}.

Connell \cite{Connell} computer-implemented the root number formulae from Rohrlich \cite{Rohr}, Liverance \cite{Liv}, Birch and Stephens \cite{BS}.

\section{Monodromy of the reduction}\label{monodromiefacile}
\subsection{For quadratic twists}
Let $E$ be the elliptic curve given by the Weierstrass equation $E:y^2=x^3+ax+b,$ where $a,b\in\Z\setminus\{0\}$ and let $\Delta$ be its discriminant, that we suppose minimal.
For every $t\in\Z\setminus\{0\}$, consider the twist of $E$ by $t$,
$E_t:y^2=x^3+at^2x+bt^3.$
Tate's algorithm allows us to show the following lemma:

\begin{lm} We have:
\begin{enumerate}
\item If $v_p(t)$ is even, then $E$ and $E_t$ have the same type of reduction at $p$.
\item If $v_p(t)$ is odd, then the type of $E_t$ and $E$ are among the following possibilities (the order is not important)
\begin{enumerate}
\item $I_0$ and $I_0^*$
\item $I_m$ and $I_m^*$
\item $II$ and $IV^*$
\item $II^*$ and $IV$
\item $III$ and $III^*$
\end{enumerate} 
\end{enumerate}
\end{lm}


\subsection{For quartic twists}

Let $E$ be the elliptic curve given by the Weierstrass equation $E:y^2=x^3+ax,$ where $a\in\Z$ is a non-zero fourth-powerfree integer.
For every $t\in\Z\setminus\{0\}$, consider the twist of $E$ by $t$:
$$E_t:y^2=x^3+atx.$$
The discriminant is $\Delta(E_t)=-2^6a^3t^3.$
Tate's algorithm allows us to show the following lemma:

\begin{lm} The reduction at $p\not=2,3$ of $E_t$ has type $I_0$, $III$, $I_0^*$, $III^*$ if $v_p(at)\equiv0,1,2,3\mod4$ respectively.
\end{lm}

\subsection{For sextic twists}

Let $E$ be the elliptic curve given by the Weierstrass equation $E:y^2=x^3+b,$ where $b\in\Z$ is a non-zero sixth-powerfree integer.
For every $t\in\Z\setminus\{0\}$, consider the twist of $E$ by $t$:
$$E_t:y^2=x^3+bt.$$
The discriminant is $\Delta(E_t)=-2^43^3b^2t^2.$
Tate's algorithm allows to show the following lemma:
\begin{lm} The reduction at $p\not=2,3$ of $E_t$ has type $I_0$, $II$, $IV$, $I_0^*$, $IV^*$, $II^*$ if $v_p(bt)\equiv0,1,2,3,4,5\mod6$ respectively.

\end{lm}

\section{Behaviour of the local root number}\label{localrootnumber}



\subsection{Local root number of a quadratic twist}

For $p\not=2,3$ a simple formula gives the local root number of an elliptic curve $E$ according to the type of its reduction:  \begin{propo}\label{Rohrlichroot} (\cite[Proposition 2]{Rohr}) Let $p\geq5$ be a rational prime, and let $E/\Q_p$ be an elliptic curve given by the Weierstrass equation $E:y^2=x^3+ax+b,$
where $(a,b)\in\Z^2\setminus(0,0)$. 
Then
\begin{displaymath}
W_p(E)=\begin{cases}
1 & \text{if the reduction of $E$ at $p$ has type $I_0$};\\
\left(\frac{-1}{p}\right) & \text{if the reduction has type $II$, $II^*$, $I_m^*$ or $I_0^*$};\\
\left(\frac{-2}{p}\right) & \text{if the reduction has type $III$ or $III^*$};\\
\left(\frac{-3}{p}\right) & \text{if the reduction has type $IV$ or $IV^*$;}\\
-\left(\frac{6b}{p}\right) & \text{if the reduction has type $I_{m}$};\\
\end{cases}
\end{displaymath}
\end{propo} 

However, it is not as simple when $p$ is $2$ or $3$.
According to \cite[1.1]{Rizz}, to determine the local root number at $p=2,3$ of an elliptic curve, we must find the smallest vector $(\alpha,\beta,\gamma)$ with nonnegative entries such that $$(\alpha,\beta,\delta)=(v_p(c_4),v_p(c_6),v_p(\Delta))+k(4,6,12)$$
for $k\in\Z$, where $c_4,c_6$ and $\Delta$ are the usual quantities associated to a Weierstrass equation (as in \cite[Chap. III]{Silv1}).

Let $E:y^2=x^3+ax+b$ be an elliptic curve. 
For every $t\in\Z\setminus\{0\}$, define $E_t$ to be the quadratic twist
\begin{displaymath}
E_t:y^2=x^3+at^2x+bt^3,
\end{displaymath}
with $a,b\in\Z$. The Weierstrass coefficients of the twisted curve are:
$$c_4=-2^4\cdot3\cdot a\cdot t^2\quad c_6=-2^5\cdot 3^3\cdot b\cdot t^3\quad \Delta=-2^4\cdot(4a^3+27b^2)t^6,$$
whence $$(\alpha,\beta,\delta)=(v_p(a),v_p(b),v_p(\Delta))+(2v_p(t),3v_p(t),6v_p(t))+\begin{cases}
(4,5,0)& \text{if $p=2$}\\
(1,3,0)&\text{if $p=3$}
\end{cases}$$

The root number is given by the  entry of Rizzo's Table II (if $p=3$) or Table III (if $p=2$) corresponding to $(\alpha,\beta,\delta)$.

\subsubsection{Periodicity}
\begin{lm} \label{signeinvariant}

For every prime $p$,
 the function $t\mapsto W_p(E_t)$ is periodic. 
\end{lm}
\begin{proof} Let $t$ be an integer, and let $t=t_0 t_{(\Delta)}$ be a decomposition such that $t_0$ is a product of prime factors of $\Delta$ (the discriminant of $E$) and $t_{(\Delta)}$ is an integer coprime with $\Delta$. Let $\Delta_{E_t}=\Delta t^6$ be the discriminant of $E_t$. 
By Tate's algorithm, the reduction at $p$ depends only on the triple $(v_p(c_4),v_p(c_6),v_p(\Delta t^6))$.

If $p\geq5$ and $v_p( t)=2$, we have $v_p(\Delta_{E_t})=v_p(\Delta)+12v_p(t_{(\Delta)})$. Since $p\nmid t_{(\Delta)}$, then $v_p(\Delta_{E_t})=v_p(\Delta)$. Yet, we know that the root number of a curve stays the same under a twist by a twelfth-power. Knowing $t$ modulo $p^2$ suffices thus to know $W_p(E_t)$.

For $p=2,3$ and $l=4,6$, let $c_{l,p}$ be the integers such that $c_l=p^{v_p(t)}c_{l,p}$. The formulae found in the tables in \cite{Halb} depend only of the 2-adic and 3-adic valuation of $c_4$ and $c_6$ as well as the remainder of $c_{4,2}$, $c_{4,3}$, $c_{6,2}$, $c_{6,3}$ modulo a certain power of 2 or 3.
\end{proof}

\subsubsection{Variation of the local root number at $p\geq5$ when twisting}
Let $p\not=2,3$ be a prime number, and $t\in\Z\setminus\{0\}$ a squarefree integer. 
In the following, we compare $W_p(E)$ and $W_p(E_{t})$.

\begin{propo}\label{variationquadratique}
Put $D_p\in\{-1,+1\}$ the integer such that $W_p(E)=D_pW_p(E_{t})$.
 We have
\begin{enumerate}
\item if $p\nmid t$
$$D_p=\begin{cases}
+1 & \text{if $E$ has good or additive reduction}\\
\left(\frac{t}{p}\right)&\text{if $E$ has multiplicative reduction}
\end{cases}$$
\item if $p\mid t$
$$D_p=\begin{cases}
+1 & \text{if $E$ has type $III$ or $III^*$}\\
\left(\frac{-1}{p}\right)&\text{if $E$ has type $I_0$ or $I_0^*$}\\
\left(\frac{3}{p}\right)&\text{if $E$ has type $II,II^*,IV,IV^*$}\\
-\left(\frac{-6b_{(p)}t_{(p)}}{p}\right)&\text{if $E$ has type $I_m$ or $I_m^*$ ($m\geq1$)}
\end{cases}$$
\end{enumerate}
\end{propo}

\begin{proof}




The value of $D_p$ depends of the type of reduction at $p$ of $E$. 

\begin{enumerate}
\item If $p\nmid t$ and the reduction of $E$ is not multiplicative, one has $D_p=+1$. If the reduction has type $I_m$, then $$D_p=\Big(\frac{t}{p}\Big),$$ because
\begin{displaymath}
W_p(E_{t})=-\Big(\frac{-6b_{(p)}t^3}{p}\Big)
\end{displaymath}
\begin{displaymath}
= -\Big(\frac{-6b_{(p)}}{p}\Big)\Big(\frac{t}{p}\Big)
\end{displaymath}
\begin{displaymath}
=W_p(E_{t})\Big(\frac{t}{p}\Big).
\end{displaymath}

\item If $p\mid t$, one of the following cases occurs, according to the type of variation of $E$ at $p$.

\begin{enumerate}
\item If $E$ has type $I_0$, then $E_{t}$ has type $I_0^*$. Conversely, if $E_t$ has type $I_0^*$, then $E_{pt}$ has type $I_0$. The local root number at $p$ changes from $(\frac{-1}{p})$ to $+1$. We have $D_p=(\frac{-1}{p})$. 
\item If $E$ has type $II$, then $E_{t}$ has type $IV^*$. Conversely, if $E_t$ has type $IV^*$, then $E_{pt}$ has type $II$. The local root number at $p$ changes from $(\frac{-1}{p})$ to $(\frac{-3}{p})$. We have $D_p=(\frac{3}{p})$.
\item If $E$ has type $III$, then $E_{t}$ has type $III^*$ and conversely. The local root number at $p$ changes from $(\frac{-2}{p})$ to $(\frac{-2}{p})$. We have $D_p=+1$.
\item If $E$ has type $I_m^*$, then $E_{t}$ has type $I_m$ and conversely. The local root number at $p$ changes from $(\frac{-1}{p})$ to $-\left(\frac{-6b_{(p)} t_{(p)} }{p}\right)$. We have $D_p=-\left(\frac{6b_{(p)} t_{(p)}}{p}\right)$.
\end{enumerate}

\begin{center}
   \begin{tabular}{ c c}
     \hline
$I_0^* \leftrightarrow I_0$ & $II \leftrightarrow IV^*$  \\
$(\frac{-1}{p})\xrightarrow{D_p=(\frac{-1}{p})} +1$ & $(\frac{-1}{p})\xrightarrow{D_p=\Big(\frac{3}{p}\Big)} (\frac{-3}{p})$ \\
&\\
$III \leftrightarrow III^*$ & $I_m^* \leftrightarrow I_m$  \\ 
$(\frac{-2}{p})\xrightarrow{D_p=+1} (\frac{-1}{p})$ & $(\frac{-1}{p})\xrightarrow{D_p=-\left(\frac{6b_{(p)}t_{(p)} }{p}\right)} -\left(\frac{-6b_{(p)}t_{(p)} }{p}\right).$\\
   \hline
\end{tabular}
 \end{center} 
\end{enumerate}



\end{proof}

\subsubsection{Local root number at $p=2$}\label{signelocauxen23}


\begin{lm}\label{signe2special}

We have $W_2(E_t)=\epsilon_2\cdot\mathrm{sgn}(t)\left(\frac{-1}{\vert t_{(2)}\vert}\right)$ (for a fixed $\epsilon_2\in\{-1,+1\}$) for all $t\in\Z$ if and only if the triple $(v_2(a)+4,v_2(b)+5,v_2(\Delta))$ is among the following:
\begin{enumerate}
\item $(0,0,0)$ or $(2,3,6)$ (then $\epsilon_2\equiv -b_{(2)}\mod4$); or
\item $(3,5,3)$ or $(5,8,9)$ and 
\begin{enumerate}
\item $a_{(2)}\equiv 3\mod8$ and $b_{(2)}\equiv3\mod4$; or $a_{(2)}\equiv 7\mod8$ and $b_{(2)}\equiv1\mod4$ (then $\epsilon_2=-1$)
\item $a_{(2)}\equiv 3\mod8$ and $b_{(2)}\equiv1\mod4$; or $a_{(2)}\equiv 7\mod8$ and $b_{(2)}\equiv3\mod4$ (then $\epsilon_2=+1$)
\end{enumerate}
\item $(\geq4,3,0)$ or $(\geq6,6,6)$ (then $\epsilon_2\equiv b_{(2)}\mod4$).
\end{enumerate}
\end{lm}

We deduce the following:

\begin{coro}
The cases listed in Theorem \ref{signe2special} are the only one such that for every $t\in\Z\setminus\{0\}$, one has $$W(E_t)=\mathrm{sgn}(t)\left(\frac{-1}{\vert t_{(2)}\vert}\right)W(E),$$
where $D_2$ is the integer depending on $t$ such that $W_2(E_t)=D_2W_2(E)$. In particular, we have $D_2=\epsilon_2\cdot\mathrm{sgn}(t)\left(\frac{-1}{\vert t_{(2)}\vert}\right)$  and so the function $w_2$ defined in Corollary \ref{decompositiontwist} is constant and $w_2(t)=\epsilon_2$.
 \end{coro}

\begin{proof}
Let $E:y^2=x^3+ax+b$ be an elliptic curve and let $\Delta$ be its discriminant. For every squarefree $t\in\Z\setminus\{0\}$, consider the twist of $E$ by $t$, $E_t:ty^2=x^3+ax+b$. 

We list here the conditions of the coefficients $a$, $b$ and $\Delta$ for the root number at 2 of the fibres $E_t$ to be $W_2(t)=\mathrm{sgn}(t)\left(\frac{-1}{ t_{(2)}}\right)$  for all positive $t$ and $-\left(\frac{-1}{-t_{(2)}}\right)$ for all negative $t$.
We can classify the curves $E$ by their triple $(v_2(a)+4,v_3(b)+5,v_2(\Delta))$: a formula for their local root number at 2 is given by the table of Rizzo \cite[Table III]{Rizz}. (In Rizzo, the notation with the $c$-invariants is preferred. Note that $(c_6)_{(2)}=3b_{(2)}$ and that $(c_4)_{(2)}=3^3a_{(2)}$.). 
The first step is to select the surfaces such that $W_2(E_t)=\epsilon\cdot\mathrm{sgn}(t)\left(\frac{-1}{\mid t\mid}\right)$ when $2\nmid t$, for a fixed $\epsilon\in\{\pm1\}$.
To find the triples and determine the additional conditions, we proceed in the following way. (We only give three examples here, the other cases being treated in a similar manner.)

If $(v_2(a)+4,v_2(b)+5,v_2(\Delta))=(0,0,0)$, then Rizzo's table gives the following formula for the root number of the fibres in odd $t$:
$$W_2(E_t)=\begin{cases}
+1 & \text{if } b_{(2)}t^3\equiv1\mod4\\
-1 & \text{otherwise.}
\end{cases}$$
We want $W_2(E_t)=+1\Leftrightarrow t\equiv1\mod4$, which is possible if and only if $b_{(2)}\equiv1\mod4$.


For some triples though, the local root number at $2$ of $E_t$ does not behave the way we want. For instance in the case of the triple $(v_2(a)+4,v_2(b)+5,v_2(\Delta))=(3,3,0)$. In this case, the local root number is 
$$W_2(E_t)=\begin{cases}
+1 & \text{if }a_{(2)}t^2\equiv3\mod4\text{ and if } b_{(2)}t\equiv\pm3\mod8,\\
 & \text{if }a_{(2)}t^2\equiv1\mod4\text{ and if } b_{(2)}t\equiv1,3\mod8,\\
-1 & \text{otherwise.}
\end{cases}
$$
In any case of $a_{(2)}\mod4$ and $b_{(2)}\mod8$, there will be values of $t\mod8$ such that $W_2(E_t)\not=\mathrm{sgn}(t)\left(\frac{-1}{\vert t\vert}\right)$. For instance, if $b_{(2)}\equiv3\mod4$ and $a_{(2)}\equiv1\mod16$, then $W_2(E_t)=+1$ if and only if $t\equiv5\mod8$.

Proceding in a similar manner for all the other triples,
we determine the triple, and the special conditions for which $t\mapsto W_2(E_t)$ behaves like $\mathrm{sgn}(t)\left(\frac{-1}{\vert t_{(2)}\vert}\right)$. We obtain the list given in Table \ref{liste1p2}. 
We have double checked every computation with the computation software MAGMA. 

\begin{center}
\begin{table}
\begin{center}
\begin{tabular}{| c | c | c |}
\hline
$(v_2(a)+4,v_2(b)+5,v_2(\Delta))$ & Special conditions & $\epsilon$, such that 
\\
(of a minimal model) &&$W_2=\epsilon\cdot \mathrm{sgn}(t)\left(\frac{-1}{\vert t\vert}\right)$
\\

\hline
$(0,0,0)$& $b_{(2)}\equiv1\mod4$ & $+1$\\ 
&$b_{(2)}\equiv3\mod4$&$-1$\\\hline

$(\geq4,3,0)$& $b_{(2)}\equiv1\mod4$,&$+1$\\ 
&$b_{(2)}\equiv3\mod4$&$-1$\\\hline
$(2,3,\geq4)$& $b_{(2)}\equiv1\mod4$,&$-1$\\ 
& $b_{(2)}\equiv3\mod4$,&+1\\ \hline
$(\geq4,4,2)$& $b_{(2)}\equiv1\mod4$,&$+1$\\ 
& $b_{(2)}\equiv3\mod4$,&$-1$\\\hline
$(3,5,3)$&
	$a_{(2)}\equiv3,5\mod8$ and $b_{(2)}\equiv3\mod4$;&$-1$\\
&$a_{(2)}\equiv1,7\mod16$ and $b_{(2)}\equiv1\mod4$,&\\ 
&	$a_{(2)}\equiv3,5\mod8$ and $b_{(2)}\equiv1\mod4$;&$+1$\\
&$a_{(2)}\equiv1,7\mod16$ and $b_{(2)}\equiv3\mod4$,&\\ 
\hline
$(5,8,9)$&
	$a_{(2)}\equiv5,7\mod8$ and $b_{(2)}\equiv3\mod4$;&$+1$\\
&$a_{(2)}\equiv1,3\mod16$ and $b_{(2)}\equiv1\mod4$,&\\ 
&	$a_{(2)}\equiv5,7\mod8$ and $b_{(2)}\equiv1\mod4$;&$-1$\\
&$a_{(2)}\equiv1,3\mod16$ and $b_{(2)}\equiv3\mod4$,&\\ 
\hline

$(\geq6,6,6)$& $b_{(2)}\equiv1\mod4$,&$+1$\\
& $b_{(2)}\equiv3\mod4$,&$-1$\\ \hline
$(4,6,7)$& $b_{(2)}\equiv3\mod4$ and $a_{(2)}\equiv7\mod8$,&$+1$\\
& $b_{(2)}\equiv1\mod4$ and $a_{(2)}\equiv7\mod8$,&$-1$
\\ \hline
$(\geq7,7,8)$& $b_{(2)}\equiv1\mod4$,&$+1$\\
& $b_{(2)}\equiv3\mod4$,&$-1$\\ \hline
$(6,8,10)$& $a_{(2)}b_{(2)}\equiv3\mod4$ &$+1$\\
& $a_{(2)}b_{(2)}\equiv1\mod4$, &$-1$\\ \hline
$(\geq7,8,10)$& $b_{(2)}\equiv1\mod 4$,&$+1$\\
& $b_{(2)}\equiv3\mod 4$.&$-1$\\
\hline
\end{tabular}
\end{center}
\caption{\label{liste1p2} Cases where $W_2(E_t)= \epsilon\cdot\mathrm{sgn}(t)\cdot\left(\frac{-1}{\vert t\vert}\right)$ for every \textbf{odd} squarefree $t\in\Z-{0}$. 
}
\end{table}
\end{center}

The final step is to retain only the triples such that $W_2(E_{2t})=W_2(E_t)$. For every case of the list, we check if the triple $(v_2(a)+2,v_2(b)+3,v_2(\Delta)+6)$ (the triple associated to a fibre at $2t$ for $t\in\Z$ odd) also gives a local root number at 2 equal to $W_2(\E_{2t})=\epsilon\cdot\mathrm{sgn}(t)\left(\frac{1}{\vert t\vert}\right)$ (for the same $\epsilon$). A list of the monodromy of the triples, as well as whether or not those pairs figure in Table \ref{liste1p2}. This way we find the cases listed in the statement of the theorem (note that for $(3,5,3)$ and $(5,8,9)$ this only works for specific classes of $a_{(2)}$ and $b_{(2)}$).

\begin{table}
\begin{center}
   \begin{tabular}{ | c | c |}
     \hline
    Monodromy of $(v_2(a)+4,v_2(b)+5,v_2(\Delta))$ & Are triples in Table \ref{liste1p2}\\ \hline
$(0,0,0) \longleftrightarrow (2,3,6)$ & yes, yes  \\
$(0,0,>0) \longleftrightarrow (2,3,>6)$ & no, no  \\
$(3,3,0) \longleftrightarrow (5,5,6)$ & no, no  \\
$(\geq4,3,0) \longleftrightarrow (\geq6,6,6)$ & yes, yes  \\
$(2,4,0) \longleftrightarrow (4,7,6)$ & no, no  \\
$(2,\geq5,0) \longleftrightarrow (4,\geq8,6)$ & no, no  \\
$(2,3,1) \longleftrightarrow (4,6,7)$ & no, no  \\
$(2,3,2) \longleftrightarrow (4,6,8)$ & no, no  \\
$(2,3,3) \longleftrightarrow (4,6,7)$ & no, no  \\
$(2,3,\geq4) \longleftrightarrow (4,6,\geq8)$ & yes, no  \\
$(3,4,2) \longleftrightarrow (5,7,8)$ & no, no  \\
$(3,5,3)\longleftrightarrow (5,8,9)$& yes, yes\\
$(4,4,2) \longleftrightarrow (6,7,8)$ & yes, no  \\
$(\geq5,4,2) \longleftrightarrow (\geq7,7,8)$ & yes, no  \\
$(4,5,4) \longleftrightarrow (6,8,10)$ & no, yes  \\
$(\geq5,5,4) \longleftrightarrow (\geq7,8,10)$ & no, yes  \\

   \hline
\end{tabular}
 \end{center}
 \caption{\label{liste2p2} }
\end{table}
\end{proof}

\subsubsection{Local root number at $p=3$}

\begin{lm}\label{signe3special}

We have $W_3(E_t)=\epsilon_3\cdot(-1)^{v_3(t)}$ for all $t\in\Z$ (for a fixed $\epsilon_3=\pm1$) if and only if the triple of values $(v_3(a)+1,v_3(b)+3,v_3(\Delta))$ is one of the following:
\begin{enumerate}
\item $(0,0,0)$,  $(1,\geq3,0)$, and $(1,2,0)$ with $a_{(3)}\equiv2\mod3$ (then $\epsilon_3=+1$)
\item $(2,3,6)$, $(3,\geq6,6)$, and $(3,5,6)$ with $a_{(3)}\equiv2\mod3$ (then $\epsilon_3=-1$)
\end{enumerate}
\end{lm}

\begin{coro}
The cases listed in Theorem \ref{signe3special} are the only ones such that for every $t\in\Z\setminus\{0\}$, one has $$W(E_t)=(-1)^{v_3(t)}W(E).$$ 

In particular, we have $D_3=\epsilon_3\cdot(-1)^{v_3(t)}W(E)$ ($D_3$ is the integer depending on $t$ such that $W_3(E_t)=D_3W_3(E)$) and so the function $w_3$ defined in Example \ref{ex1} or Corollary \ref{decompositiontwist} is constant and $w_3(t)=\epsilon_3$.
\end{coro}

\begin{proof}

Let $t\in\Z\setminus\{0\}$ be an integer.

Suppose that $t$ satisfies $3\nmid t$. Then the local root number at 3 is given by the entry of \cite[Table II]{Rizz} corresponding to $(v_3(a)+1,v_3(b)+3,v_3(\Delta))$.
If $v_3(t)$ is odd, then the triple is $(v_3(a)+3,v_3(b)+6,v_3(\Delta)+6)$ and the corresponding entry gives the local root number. We double-checked every computation using the software MAGMA.

We start by selecting the triples such that $W_3(E_t)=W_3(E)=\epsilon_3$ for all $t\in\Z$ prime to 3 (for some fixed $\epsilon_3\in\{\pm1\}$). The list is given in Table \ref{liste1p3}.

\begin{table}
\begin{center}
\begin{tabular}{ | c | c | c |}
\hline
$(v_3(a)+1,v_3(b)+3,v_3(\Delta)$ & Special conditions & $\epsilon_3$, such that $W_3=\epsilon_3$ \\ \hline
  $(0,0,0)$, &&$+1$\\ 
  \hline
$(1,\geq3,0)$,&& $+1$\\ \hline
$(1,2,0)$,&& $+1$\\ \hline
$(2,\geq5,3)$&& $+1$,\\ \hline
$(2,3,3)$ &if $a_{(3)}\equiv1\mod3$ and $b_{(3)}\equiv4,5\mod9$ &$+1$, \\ \hline
$(2,3,4)$ && $+1$,\\ \hline
$(2,3,6)$ &&$-1$\\ \hline
$(2,3,\geq7)$&&$-1$\\ \hline
$(2,3,4)$&&$+1$\\
$(2,\geq5,3)$&&+1\\ \hline
$(\geq3,3,3)$& $b_{(2)}\equiv1,8\mod9$&$+1$\\ \hline
$(3,5,6)$ & if $a_{(3)}\equiv1\mod3$&$+1$\\ 
&$a_{(3)}\equiv2\mod3$&$-1$\\ \hline
$(3,\geq6,6)$&&$-1$\\ \hline
$(4,6,9)$ & if $a_{(3)}\equiv1\mod3$ and $b_{(3)}\equiv4,5\mod9$ &$+1$,\\  \hline
$(4,\geq8,9)$ &&+1\\ \hline
$(\geq5,6,9)$& $b_{(2)}\equiv1,8\mod9$&$+1$\\ \hline
\end{tabular}
 \end{center}
\caption{\label{liste1p3} Cases where $W_3(E_t)= \epsilon_3$ for every squarefree $t\in\Z\setminus{0}$ not divisible by $3$. 
}
\end{table}


Now, we check among those triples whether they have the property that $W(E_{3t})=-W(E_t)$, by verifying that both the triple of $E$ and the triple of a twist by $3$ are in Table \ref{liste1p3} (these facts are listed in Table \ref{liste2p3}) and comparing the values of $\epsilon_3$. 
We retain only 
$(0,0,0)$, $(2,3,6)$,
 $(1,\geq3,0)$ $(3,\geq6,6)$,
  and finally $(1,2,0)$ and $(3,5,6)$ with the additional condition that $a_{(3)}\equiv2\mod3$. It is important to notice that in the following cases, although the triple $(v_3(a)+1,v_3(b)+3,v_2(\Delta))$ and the triple of a quadratic twist by $3$ are both in Table \ref{liste1p3}, we have $W_3(E_t)=W_3(E_{3t})$ rather than $W(E_t)=-W(E_{3t})$ as required:
 $(2,3,3)$ and $(4,6,9)$ with the additional condition that $a_{(3)}\equiv1\mod3$ and $b_{(3)}\equiv4,5\mod9$,
 and $(\geq3,3,3)$ and $(\geq5,6,9)$ with the additional condition that $b_{(3)}\equiv\pm1\mod9$.

A yes in the second column of Table \ref{liste2p3} means that the triple is in Table \ref{liste1p3}, but under special conditions. If there is no special condition to take account of, they the table tells the value of the function $W_3(\E_t)$.

\begin{table}
\begin{center} 
   \begin{tabular}{ | c | c |}
     \hline
    Monodromy of $(v_3(a)+1,v_3(b)+3,v_3(\Delta))$ & Are triples in Table \ref{liste1p3}\\ \hline
$(0,0,0) \longleftrightarrow (2,3,6)$ & $+1$, $-1$  \\
$(0,0,>0) \longleftrightarrow (2,3,>6)$ & no, $-1$  \\
$(1,2,0) \longleftrightarrow (3,5,6)$ & $+1$, yes  \\
$(1,\geq3,0) \longleftrightarrow (3,\geq6,6)$ & $+1$, $-1$  \\
$(\geq2,2,1) \longleftrightarrow (\geq4,5,7)$ & no, no  \\
$(2,3,3) \longleftrightarrow (4,6,9)$ & yes, yes  \\ 
$(\geq3,3,3) \longleftrightarrow (\geq5,6,9)$ & yes, yes  \\
$(2,4,3) \longleftrightarrow (4,7,9)$ & no, no  \\. 
$(2,\geq5,3) \longleftrightarrow (4,\geq8,9)$ & $+1$, $+1$  \\
$(2,3,4) \longleftrightarrow (4,6,10)$ & $+1$, no  \\
$(2,3,5) \longleftrightarrow (4,6,11)$ & no, no  \\
$(\geq3,4,5)\longleftrightarrow (\geq5,7,11)$& no, no\\
   \hline
\end{tabular}
 \end{center}
 \caption{\label{liste2p3}}
\end{table}

\end{proof}

\subsection{Root number of a quartic twist}

\begin{lm}\cite[Lemme 4.7]{VA}\label{VAx}

Let $t$ be a non-zero integer and define the elliptic curve $E_t:y^2=x^3+tx$. We denote by $W_2(t)$ and $W_3(t)$ its local root numbers at 2 and 3. Let $t_{(2)}$ 
be the integer such that $t=2^{v_2(t)}t_{(2)}$. 
 Then
\begin{displaymath}
W_2(t)= 
\begin{cases}
 -1 & \text{if }v_2(t)\equiv1\text{ or } 3\mod4 \text{ and }t_{(2)}\equiv1\text{ or }3\mod8; \\
  & \text{or if }v_2(t)\equiv0\mod4 \text{ and } t_{(2)}\equiv1,5,9,11,13\text{ or }15\mod16;\\
   & \text{or if }v_2(t)\equiv2\mod4 \text{ and }t_{(2)}\equiv1,3,5,7,11,\text{ or }15\mod16;\\
 +1 & \text{otherwise.}
\end{cases}
\end{displaymath}

\begin{displaymath}
W_3(t)=
\begin{cases}
 -1 & \text{if }v_3(t)\equiv2\mod4; \\
 +1 & \text{otherwise.}
\end{cases}
\end{displaymath}
\end{lm}

We may reformulate this theorem as follows:

\begin{lm}\label{j1728p2} The function of the local root number at $2$ on the fourth-powerfree integers, 
defined as $t\mapsto W_2(\E_t)$, is $2^4$-periodic on $\F_4$ and it takes the value $-1$ if and only if $t\in[2^ka]$ the congruence class mod $2^4$, where $k=1,3\mod4$ and $a\equiv1,3\mod8$; or $k=0$ and $a\equiv1,5,9,11,13,15\mod16$; or $k=2$ and $a\equiv1,3,5,7,11,15\mod16$.
\end{lm}

\begin{lm}\label{j1728p3}The function of the local root number at $3$ on the fourth-powerfree integers, 
defined as $t\mapsto W_3(\E_t)$, is $3^4$-periodic on $\F_4$ and it takes the value $-1$ if and only if $t\in[3^2a]$ the congruence class mod $3^4$, where $a\in\Z$.
\end{lm}

\subsection{Root number of a sextic twist}
\begin{lm}\label{VA} \cite[Lemme 4.1]{VA}

Let $t$ be a non-zero integer and the elliptic curve $E_t:y^2=x^3+t$. We denote by $W_2(t)$ and $W_3(t)$ its local root numbers at 2 and 3. Put $t_{(2)}$ and $t_{(3)}$ the integers such that $t=2^{v_2(t)}t_{(2)}=3^{v_3(t)}t_{(3)}$. Then
\begin{displaymath}
W_2(t)= 
\begin{cases}
 -1 & \text{if }v_2(t)\equiv0\text{ or } 2\mod6; \\
  & \text{or if }v_2(t)\equiv1,3,4\text{ or }5\mod6\text{ and }t_{(2)}\equiv3\mod4;\\
 +1, & \text{otherwise.}
\end{cases}
\end{displaymath}

\begin{displaymath}
W_3(t)=
\begin{cases}
 -1 & \text{if }v_3(t)\equiv1\text{ or } 2\mod6\text{ and }t_{(3)}\equiv1\mod3; \\
  & \text{or if }v_3(t)\equiv4\text{ or }5\mod6\text{ and }t_{(3)}\equiv2\mod3;\\
  & \text{or if }v_3(t)\equiv0\mod6\text{ and }t_{(3)}\equiv5\text{ or }7\mod9;\\
  & \text{or if }v_3(t)\equiv3\mod6\text{ and }t_{(3)}\equiv2\text{ or }4\mod9;\\
 +1, & \text{otherwise.}
\end{cases}
\end{displaymath}
\end{lm}

We may reformulate this theorem as follows:
\begin{lm}\label{j0p2}The function of the local root number at $2$ defined as $t\mapsto W_2(\E_t)$ is $2^6$-periodic on $\F_6$ and it takes the value $+1$ if and only if $t\in[2^ka]$ the congruence class mod $2^6$, where $k\in\{1,3,4,5\}$ and $a\equiv3\mod4$.
\end{lm}

\begin{lm}\label{j0p3}The function of the local root number at $3$ defined as $t\mapsto W_3(\E_t)$ is $3^6$-periodic on $\F_6$ and it takes the values $-1$ if and only if $t\in[3^ka]$ the congruence class mod $3^6$, where $k\in\{1,2\}$ and $a\equiv1\mod3\mod4$; $k\in\{4,5\}$ and $a\equiv2\mod3\mod4$; $k=0$ and $a\equiv5,7\mod3\mod4$; or $k\in\{1,2\}$ and $a\equiv2,4\mod3\mod4$.
\end{lm}


\section{Proof of the main theorem}\label{proof}

\subsection{Families of quadratic twists}

\begin{thm}\label{thmquadratic} Let $E$ be an elliptic curve such that $j(E)\not=0,1728$.
Define $\mathscr{F}_2$ to be the set of squarefree integers.
Then

\begin{enumerate}
\item the function $t\mapsto W(\E_t)$ is periodic on $\mathscr{F}_2$,
\item the root number $W(\E_t)$ is not constant when $t$ runs through $\mathbb{Z}\setminus \{0\}$.
\end{enumerate}
\end{thm}

The proof of this theorem is based on the following lemma of independent interest.

\begin{lm}\label{brot}
Let $E$ be an elliptic curve.

Suppose that for all positive squarefree integer $t\in\N$ one has $W(E_t)=W(E)$. This happens if and only if the elliptic curve $E$ has the following properties:
\begin{enumerate}[(a)]
\item there is no finite place of multiplicative reduction except possibly at 2 or 3;
\item for all $t$ non-zero squarefree integers one has $W_2(E_t)=W_2(E)\cdot\mathrm{sgn}(t)\cdot\left(\frac{-1}{\vert t_{(2)}\vert}\right)$;
\item for all $t$ non-zero squarefree integers one has $W_3(E_t)=W_3(E)\cdot(-1)^{v_3(t)}$;
\item if there exists $p\mid\Delta_{(6)}$ then:
\begin{enumerate}[i.]\item if the reduction of $E$ at $p$ has type $III$ or $III^*$, then $p\equiv1\mod4$;\item if the reduction of $E$ at $p$ has type $II$, $II^*$, $IV$ or $IV^*$, then $p\equiv1\mod6$;\item the reduction at $p$ is not additive potentially multiplicative. \end{enumerate}
\end{enumerate}
\end{lm}

\begin{rem}
The elliptic curves $E$ with property (b) or (c) are respectively given by Lemma \ref{signe2special} or \ref{signe3special}.
\end{rem}

\begin{ex}
Let $E_{+1}:y^2=x^3-91x+182$ and $E_{-1}:y^2=x^3-91-182$. For any $t\in\N$, the quadratic twist of $E_{+1}$ by $t$ has root number $+1$ and the quadratic twist of $E_{-1}$ has root number $-1$. Indeed, the discriminant of these elliptic curves is $\Delta=2^{12}7^213^2$. So we have:
\begin{enumerate}[(a)] \item The only finite place of multiplicative reduction is in $2$.
\item By Lemma \ref{signe2special}, $W_2(E_t)=W_2(E)\cdot\mathrm{sgn}(t)\left(\frac{-1}{\vert t_{(2)}\vert}\right)$ because the associated triple of a minimal model of $E_t$ is $(v_2(c_4),v_2(c_6),v_2(\Delta))=(0,0,0)$.
\item By Lemma \ref{signe3special}, $W_3(E_t)=W_3(E)\cdot(-1)^{v_3(t)}$ because the associated triple of a minimal model of $E_t$ is $(v_2(c_4),v_2(c_6),v_2(\Delta))=(1,3,0)$ and $a_{(3)}\equiv2\mod3$.
\item The only prime dividing $\Delta_{(6)}$ are $7$ and $13$, for both of them the reduction has type $II$ and that $p\equiv1\mod6$
\end{enumerate}

\end{ex}

\begin{proof}[Proof of Theorem \ref{thmquadratic}.]
Let $E_t$ be the elliptic curve obtained by the twist by $t\in\Z$ of the elliptic curve $E:y^2=x^3+ax+b$ with integer coefficients $a,b\in\Z$. Put $\Delta=\Delta(E)=2^{v_2(\Delta)}3^{v_3(\Delta)}\Delta_{(6)}$ the discriminant of $E$ that we suppose minimal.

The root number can be written as $$W(E_t)=-W_2(E_t)W_3(E_t)\prod_{p\mid t; p\nmid6\Delta}{W_p(E_t)}\prod_{p\mid\Delta_{(6)}}{W_p(E_t)}.$$
The places at $p\mid t$ such that $p\nmid\Delta$ have reduction type $I_0^*$ on $E_t$, and thus by Proposition \ref{Rohrlichroot} one has $W_p(E_t)=\left(\frac{-1}{p}\right)$.
$$W(E_t)=-W_2(E_t)W_3(E_t)\left(\frac{-1}{\vert t_{(6\Delta)}\vert}\right)\left(\prod_{p\mid\Delta_{(6)}}{W_p(E_t)}\right).$$

We have $\Delta(E_t)=\Delta t^6$. Since $t$ is assumed squarefree, the equation of $E_t$ is minimal unless $\gcd(\Delta,t)\not=1$.\\



\emph{Proof of the first property. }

For each prime $p$, Lemma \ref{signeinvariant} proves that there exists an integer $\alpha_p$ such that $W_p(E_t)=W_p(E_{t'})$ for all $t\equiv t'\mod p^{\alpha_p}$. Put $N=2^{\alpha_2}3^{\alpha_3}\prod_{p\mid\Delta_{(6)}}{p^{\alpha_p}}$. The global root number is the same for $t$ and $t'$, squarefree integers in the same congruence class modulo $N$.\\



\emph{Proof of the second property.} 

It is now left to show that this partition is not trivial, in other words that there exists at least one congruence class $[t_+]$ with root number $+1$ and one class $[t_-]$ with root number $-1$.

By Lemma \ref{brot}, we know that if $E$ does not satisfy one or more of the properties (a) to (d), then there exists such $t_+,t_-\in\N$. Suppose then that $E$ satisfies the properties (a) to (d). Then by Lemma \ref{brot}, $W(E_t)=W(E)$ for any squarefree positive $t\in\N$. We compare the twists $E_1(=E)$ and $E_{-1}$, then by Proposition \ref{variationquadratique} combined with the fact that $E$ has no multiplicative reduction, for $p\not=2,3$ we have $W_p(E_{-1})=W_p(E)$. The global root numbers are thus related as follows:
$$W(E_{-1})=-W_2(E_{-1})W(E_{-1})\prod_{p}{W_p(E_1)}$$
Observe that by assumption of property b), $W_2(E_{-1})=W_2(E)\cdot \mathrm{sgn}(-1)\left(\frac{-1}{\vert -1\vert}\right)=-W_2(E)$ (by convention, $\left(\frac{-1}{1}\right)=1$). Moreover by assumption of property c), $W_3(E_{-1})=(-1)^{v_3(-1)}W_3(E)=W_3(E)$. Thus
$$W(E_{-1})=-W_2(E)W_3(E)\prod_{p\not=2,3}{W_p(E)}=W(E).$$

\end{proof}

\begin{proof}[Proof of Lemma \ref{brot}.]

By Theorem \ref{thmquadratic}'s first property, the root number is periodic on the squarefree integers. In other words there exists a positive integer $N$ such that for $t\equiv t'\mod N$ we have $W(E_t)=W(E_{t'})$. Hence in particular any $t\equiv 1\mod N$ will have the same global root number as $W(E)$.
We look for a congruence class $[t_-]\mod N$ such that $W(E_{t_-})=-W(E)$. 

Let $t$ be a non-zero squarefree integer.
For each prime $p$, let $D_p\in\{+1,-1\}$ be the integer depending on $t$ such that $W_p(E_{t})=D_p W_p(E)$.  In Proposition \ref{variationquadratique}, we already computed the values of $D_p$ (for $p\not=2,3$) according to the reduction of $E$ at $p$ and to whether $p$ divides $t$.

One has $$W(E_t)=-\prod_{p\mid \Delta t}{\left(D_pW_p(E)\right)}=\left(\prod_{p\mid \Delta t}{D_p}\right)W(E),$$
which we can split the following way:
$$W(E_t)=D_2D_3\left(\prod_{p\mid\Delta_{(6)}; p\nmid t}{D_p}\right)\left(\prod_{p\mid t}{D_p}\right)W(E)$$
$$=D_2D_3\left(\prod_{ p\nmid 6t \atop E\text{ has mult. red. at p}}{D_p}\right)\left(\prod_{ p\nmid 6t \atop E\text{ has add. red. at p}}{D_p}\right)\left(\prod_{p\mid t\atop E\text{ has type }I_0}{D_p}\right)\left(\prod_{p\mid t\atop E\text{ not $I_0$}}{D_p}\right)W(E)$$
and Proposition \ref{variationquadratique} gives
$$=D_2D_3\left(\prod_{ p\nmid 6t \atop E\text{ has mult. red. at p}}{\left(\frac{t}{p}\right)}\right)\left(\prod_{p\mid t\atop E\text{ has type }I_0}{\left(\frac{-1}{p}\right)}\right)\left(\prod_{p\mid t\atop E\text{not $I_0$}}{D_p}\right)W(E)$$
so that

 $$W(E_{t})=D_2D_3\Big(\frac{t}{\Delta_M}\Big)\left(\frac{-1}{\vert t_
{(6\Delta)}\vert}\right)\left(\prod_{p\mid t \atop \text{ $E$ has bad red.}}{D_{p_0}}\right)W(E_t),$$
where $\Delta_M$ is the product of the prime numbers $p\not=2,3$ at which the reduction of $E$ has multiplicative reduction. Let $t_{(6\Delta)'}$ be the integer such that 
$$t=\mathrm{sgn}(t)\cdot2^{v_2(t)}3^{v_3(t)}t_{(6\Delta)'}t_{(6\Delta)},$$ 
that is to say $t_{(6\Delta)'}=\prod_{p\mid\Delta_{(6)}}{p^{v_p(t)}}$. 
We can write 
$$\Big(\frac{-1}{\vert t_{(6\Delta)}\vert }\Big)=\mathrm{sgn}(t)\left(\frac{-1}{\vert t_{(2)}\vert}\right)(-1)^{v_3(t)}\left(\frac{-1}{\vert t_{(6\Delta)'}\vert}\right).$$ 
Let $$C_2=\mathrm{sgn}(t)\cdot D_2\left(\frac{-1}{\vert t_{(2)}\vert}\right)\text{, }C_3=D_3(-1)^{v_3(t)}\text{, }C_M=\Big(\frac{t}{\Delta_M}\Big)\text{, }C_0=\prod_{p\mid t\atop\text{ $E$ has add. red.}}{D_{p}}\text{, }
C_{\Delta}=\left(\frac{-1}{t_{(6\Delta)'}}\right)$$
 so that
 $$C=C_2C_3C_MC_0C_{\Delta},$$ is the integer depending on $t$ such that $W(E_{ t})=C\cdot W(E)$. 
 In the remainder of the proof, we study in more detail the variation of each component in dependance of $t$.




Observe that it is always possible to find $t$ such that $C_2=+1$ (or $C_3=+1$), a trivial example being $t=1$.
We now look at multiple ways to obtain $C=-1$ for a certain positive integer $t$, depending on the properties of $E$.

(a) Suppose $E$ admits places of multiplicative reduction, i.e. $\Delta_M\not=1$. Let $t_0$ be a prime number with the following properties:
\begin{itemize} \item it does not divide $\Delta_{(6)}$, (so that $C_\Delta=C_0=+1$)\item it is not a square modulo $\Delta_M$ (so that $C_M=-1$) \item and it is such that $C_2=+1$ and $C_3=+1$.\end{itemize} Such a prime exists by the Chinese Remainder Theorem. For this $p_0$, we have $$W(E_{p_0})=-W(E).$$

However, in case $\Delta_M=1$, this construction is not possible, since one has $C_M=+1$ for any squarefree integer $t$.  Suppose thus that $E$ has no multiplicative reduction at any $p\not=2,3$ and let us look at what we can do instead.

%
Suppose that there exists a positive $t_0\nmid\Delta_{(6)}$ such that for the twist $E_{t_0}$ one has $C_2=-1$, then by the Chinese Remainder Theorem we can choose a positive $t \nmid\Delta_{6}$ (in addition we may assume that it is squarefree) respecting $t\equiv t_0\mod2^{\alpha_2}$ and such that $C_3=C_M=C_0=+1$. For this $t$, one has $W(E_t)=-W(E).$
In a similar way, if there exists a positive $t_0'\nmid\Delta$ such that $C_3=-1$, then one can find a $t'\in\N$ such that $W(E_t')=-W(E)$.

Thus it is necessary for the triple $(v_p(c_4),v_p(c_6),v_p(\Delta))$ to be among those listed in Lemma \ref{signe2special} when $p=2$ and among those listed in the Lemma \ref{signe3special} when $p=3$, so that they have the property that $C_2=+1$ and $C_3=+1$ for any squarefree $t\in\N$.

(d)
In the following, we will assume that $C_2=C_3=C_M=+1$ for any quadratic twist of $E$. Let us look at $E_{p_0}$ the twist at a prime $p_0\not=2,3$ that divides $\Delta$.
According to the reduction at $p_0$, we have different a value of $C_0=D_{p_0}$ given by Proposition \ref{variationquadratique}. As for $C_\Delta$, it is equal to $\left(\frac{-1}{p_0}\right)$. 
\begin{enumerate}[i.]
\item If the reduction is of type $III$ or $III^*$, then $D_{p_0}=+1$ (i.e. $W_{p_0}(E_{t})=W_{p_0}(E)$ if $p_0\mid t$), so for $E_{p_0}$ we have $C_0=+1$ and $C_\Delta=\left(\frac{-1}{p_0}\right)$. Thus $W(E_{p_0})=\left(\frac{-1}{p_0}\right)W(E)$ for all $t\in\Z\setminus\{0\}$. If $p_0\equiv3\mod4$, then $W(E_{p_0})=-W(E)$, otherwise $W(E_{p_0})=W(E)$.

\item If the reduction has type $II$, $II^*$, $IV$ or $IV^*$ then $D_{p_0}=\left(\frac{3}{p_0}\right)$, and thus $W(E_{p_0})=\left(\frac{3}{p_0}\right)\left(\frac{-1}{p_0}\right)W(E_t)=\left(\frac{-3}{p_0}\right)W(E_t)$ for all $t\in\Z\setminus\{0\}$. If $p_0\equiv5\mod6$, then $W(E_{p_0})=-W(E)$, otherwise if $p_0\equiv1\mod6$, $W(E_{p_0})=W(E)$. 

\item If the reduction has type  
$I_m^*$, then 
in the summerizing table of Section \ref{monodromiefacile} we find that $D_{p_0}=-\left(\frac{-6b_{(p_0)}}{p_0}\right)$. 
Put $r\in\{\pm1\}$ to be the remainder modulo $4$ of $p_0$. Then take $t$ such that $t\equiv -6b_{(p_0)}r \mod p_0$ and prime to $6\Delta$. 
We have
\begin{align*}
W(E_{p_0})&=\left(\frac{-1}{p_0}\right)D_0W(E),& \text{since }C_2=C_3=C_M=+1\nonumber\\
&=\left(\frac{-1}{p_0}\right)\left(\frac{-6b_{(p_0)}t}{p_0}\right)W(E)=\left(\frac{-r}{p_0}\right)\left(\frac{(-6b_{(p_0)})^2}{p_0}\right)W(E)&\nonumber\\
&=-W(E).&\\
\end{align*}
\item If the reduction has type $I_0^*$, then $D_{p_0}=\left(\frac{-1}{p_0}\right)$, and thus $W(E_{p_0})=W(E)$ in any case.
\end{enumerate}

Observe that if we take, rather than a prime, a (positive) squarefree product $t_0=\prod{p_0}$ of primes $p_0$ dividing $\Delta_{(6)}$, then $C_0=\prod_{p_0}{D_{p_0}}$ and $C_{\Delta}=\left(\frac{-1}{\prod{p_0}}\right)$. So if each of the condition described in i. ii.  iii. and iv. hold on the corresponding factor of $t_0$, the root number is $W(E_{t_0})=W(E)$. 

Suppose that 
the hypotheses a), b), c) and d) in the statement are verified. Let $t$ be a positive integer. Then by hypotheses a), $C_M=+1$, by b) and c), $w_2=w_3=+1$, and by d), the contributions of a $p_0$ to $C_0$ and $C_\Delta$ (those are respectively equal to $D_{p_0}$ and $\left(\frac{-1}{p_0}\right)$) are compensating each other: $C_0\cdot C_\Delta=+1$. Thus $W(E_t)=W(E)$.
This prove Lemma \ref{brot}.
\end{proof}

As a corollary of the proofs of Theorem \ref{thmquadratic} and of Lemma \ref{brot}, we have the following decomposition that become handy in the computation of examples (see Examples \ref{ex1} to \ref{ex3}):

\begin{coro}\label{decompositiontwist}
Let $E$ be an elliptic curve, and $E_t$ the quadratic twist by $t\in\mathscr{F_2}$. Then the root number can be written as the following product:
$$W(E_{t})=-\prod_{p=2,3, \text{ $p$ bad reduction}}{w_p(t)},$$
where $w_p(t)$ are "local contributions" at each $p$ determined as follows:
$$w_p(t)=\begin{cases}\mathrm{sgn}(t)D_2\left(\frac{-1}{\mid t_{(2)}\mid}\right)W_2(E)& \text{for }p=2\\ (-1)^{v_3(t)}D_3W_3(E)&\text{for $p=3$}\\ D_p\left(\frac{-1}{p}\right)^{v_p(t)}W_p(E)&\text{for $p$($\not=2,3$) of additive reduction}\\ 
\left(\frac{t_{(173)}}{173}\right)D_pW_p(E)&\text{for $p$($\not=2,3$) of multiplicative reduction,}
\end{cases}$$ 
where $D_p$ are the integers such that $W_p(E_t)=D_pW_p(E)$ (they are given by Proposition \ref{variationquadratique} for $p\not=2,3$).

\end{coro}



\subsection{Families of sextic twists}

For any $t\in\Z\setminus\{0\}$, put the curve $E_t:y^2=x^3+t$.

\begin{thm} \label{thmA}

Let $t\in\mathscr{F}_6$ be a sixth-powerfree integer.
Then the root number can be written as
\begin{displaymath}
W(E_t)=-W_2(E_t)W_3(E_t)\left(\frac{-1}{\vert t_{(6)}\vert}\right)\left(\frac{sq(t)_{(6)}}{3}\right).
\end{displaymath}

Moreover, it is $(2^63^6,3)$-square periodic and
there exists square-congruence classes of $t$ such that $W(E_t)=+1$, and another such that $W(E_t)=-1$.
\end{thm}

\begin{proof}

Let $E_t:y^2=x^3+t$ be a family of elliptic curve parametrized by the variable $t$.

Suppose $t$ sixth-powerfree. The root number can be written as
\begin{displaymath}
W(E_t)=-W_2(E_t)W_3(E_t)\prod_{p\vert t\text{, }p\not=2,3}{W_p(E_t)}.
\end{displaymath}
As previously, we simplify the notations by writting $W_p(t)$ rather that $W_p(E_t)$. The formulae for the local root number at $2$ and $3$ are given by \cite[Lemme 4.1]{VA}. 
The local root number at $2$ stays the same on a congruence class modulo $2^6$ as proven in Lemma \ref{j0p2} and the local root number at $3$ stays the same on a congruence class modulo $3^6$ as proven in \ref{j0p3}.

Now look at the product in the root number formula, i.e. the quantity $$\mathscr{P}(t):=\prod_{p\vert t\text{, }p\not=2,3}{W_p(E_t)}.$$

By \cite[Prop. 2]{Rohr} (reported in Proposition \ref{Rohrlichroot}), we have
\begin{displaymath}
\mathscr{P}(t)=\prod_{p\vert t, p\not=2,3}
\begin{cases}
 +1 & \text{si }v_p(t)\equiv0\mod6; \\
 \big(\frac{-1}{p}\big) & \text{si } v_p(t)\equiv1,3,5\mod6; \\
 \big(\frac{-3}{p}\big) & \text{si } v_p(t)\equiv2,4\mod6.
\end{cases}
\end{displaymath}

Since $t$ is assumed sixth-powerfree, we can write it as $t=2^{\alpha}3^{\beta}t_1t_2^2t_3^3t_4^4t_5^5$ where $t_i$ are pairwise prime and neither divisible by $2$ nor $3$. Note that $\vert t_{(6)}\vert =t_1t_2^2t_3^3t_4^4t_5^5$ and $\mathrm{sq}(t)_{(6)}=t_2t_4$.

Then we split $\mathscr{P}(t)$ in two parts, according to whether $p\vert\tau_1$ or $p\vert \tau_2$.

\begin{displaymath}
\mathscr{P}(t)=\prod_{p\vert \tau_1}{W_p(E_t)}\prod_{p\vert \tau_2}{W_p(E_t)}
\end{displaymath}

\begin{displaymath}
=\prod_{p\vert \tau_1}{\left(\frac{-1}{p}\right)}\prod_{p\vert \tau_2}{\left(\frac{-3}{p}\right)}
\end{displaymath}

\begin{displaymath}
=\left(\frac{-1}{\tau_1}\right)\left(\frac{-3}{\tau_2}\right)
=\left(\frac{-1}{\tau_1}\right)\left(\frac{\tau_2}{3}\right)
\end{displaymath}

We obtain that the root number is
\begin{displaymath}
W(E_t)=-W_2(E_t)W_3(E_t)\left(\frac{-1}{\tau_1}\right)\left(\frac{\tau_2}{3}\right).
\end{displaymath}

The root number depends of the remainders of $t$ modulo $2^73^7$ and of $\tau_2$ modulo 3. Note that it is possible to find a pair of congruence classes $[t_+]$ and $[t_-]$ such that $W(E_{t_+}=-W(E_{t_-})$. Indeed, if we take $t_+\equiv C\mod 2^73^7$, we can (by replacing $\tau_2$ by an integer $\tau_2'$ prime to $6\tau_1$ and such that $\tau_2'\equiv-\tau_2\mod 3$) define $t_-=2^\alpha3^\beta t_1\tau_2'$ whose root number is opposite to the root number of $t_+$.
\end{proof}


\begin{rem}
In \cite{VA}, V\'arilly-Alvarado proved that the root number varies on a surface of the form $y^2=x^3+AT^6+B$ except possibly if $3A/B$ is a rational square. In \cite{Desjardins2} the author explicit the conditions on $A$ and $B$ for the elliptic surfaces on which the root number is constant along the fibration. In particular, we give here an example that was previously found by V\'arilly-Alvarado \cite{VA}:
\begin{equation}\label{s4}
y^2=x^3+27t^6+16.
\end{equation}


Consider the elliptic surface given by (\ref{s4}). Every fibre $E_t$ at $t=\frac{m}{n}\in\Q$ is $\Q$-isomorphic to the curve:
$$E_{m,n}:y^2=x^3+27m^6+16n^6.$$
Remark that $-3\equiv(4m^3/3n^3)^2\mod p$ and thus $\Big(\frac{-3}{p}\Big)=1$ for all $p\vert f(t)$. The formula becomes 
\begin{displaymath}
W(E_t)=-W_2(E_t)W_3(E_t)\left(\frac{-1}{(27m^6+16n^6)_{(2)}}\right).
\end{displaymath}

Moreover, for every value of $t$, the remainder modulo $3$ of $27m^6+16n^6$ is $1$. The minimal value of $M_2$ stays $3$.
It is possible to find a value of $M_1$ lower that the one given by the theorem (which is $M_1=2^73^7$). This new value is $M_1'=2^43^0=16$. 
Observe that $27m^6+16n^6$ falls in the congruence classes $1,11\mod 16$. These values correspond to a root number $+1$. However, the generic rank can be proven to be equal to $2$ on this elliptic surface, so the set of rational points is Zariski-dense.
\end{rem}

\subsection{Families of quartic twists}

For every $t\in \Z\setminus\{0\}$, let $E$ be the curve  $E_t:y^2=x^3+tx$.

\begin{thm}\label{thmB}

Let $t\in\mathscr{F}_4$ be a fourth-powerfree integer.
Then the root number $E_t:y^2=x^3$ can be written as
\begin{displaymath}
W(E_t)=-W_2(E_t)W_3(E_t)\left(\frac{-2}{\vert t_{(6)}\vert}\right)\left(\frac{-1}{sq(t)_{(6)}}\right).
\end{displaymath}

Moreover, the function $t\mapsto W(E_t)$ is $(2^4,4)$-square periodic and
there exists classes for which $W(E_t)=+1$, and other such that $W(E_t)=-1$.
\end{thm}

\begin{proof}

Let $t\in\mathscr{F}_4$ and define the elliptic curve $E_t:y^2=x^3+tx$. The root number of $E_t$ can be written as
\begin{displaymath}
W(E_t)=-W_2(E_t)W_3(E_t)\prod_{p\vert t\text{, }p\not=2,3}{W_p(E_t)}.
\end{displaymath}
By Lemmas \ref{j1728p2} and \ref{j1728p3}, the local root number at $2$ and $3$ are periodic.
Look now at the remaining part of the root number, namely
\begin{equation}\label{produitpdivt}
\prod_{p\vert t, p\not=2,3}{W_p(E_t)}.
\end{equation}

By \cite[Prop. 2]{Rohr}, we have
\begin{displaymath}
\prod_{p\vert t, p\not=2,3}{W_p(E_t)}=\prod_{p\vert t, p\not=2,3}
\begin{cases}
 +1 & \text{if }v_p(t)\equiv0\mod4; \\
 \big(\frac{-1}{p}\big) & \text{si } v_p(t)\equiv2\mod4 \\
 \big(\frac{-2}{p}\big) & \text{si } v_p(t)\equiv1,3\mod4.
\end{cases}
\end{displaymath}

Since $t$ is assumed fourth-powerfree, we can write it as $t=2^{\alpha}3^{\beta}t_1t_2^2t_3^3$, where $t_i$ are pairwise coprime and not divisible by $2$ nor $3$. Note that $\vert t_{(6)}\vert=t_1t_2^2t_3^3$ and that $\mathrm{sq}(t)=t_2$.
We split (\ref{produitpdivt}) in two parts according to whether $p\mid t_1 t_3$ or $p\vert t_2$.

\begin{align*}
\prod_{p\vert t}{W_p(E_t)}&=\prod_{p\vert t_1 t_3}{W_p(E_t)}\prod_{p\vert t_2}{W_p(E_t)}\nonumber \\
&=\prod_{p\vert \tau_1}{\left(\frac{-2}{p}\right)}\prod_{p\vert t_2}{\left(\frac{-1}{p}\right)}\nonumber\\
&=\Big(\frac{-2}{t_1t_3}\Big)\Big(\frac{-1}{t_2}\Big).\nonumber\\
&=\Big(\frac{-2}{\vert t_{(6)}\vert}\Big)\Big(\frac{-1}{\mathrm{sq}(t)_{(6)}}\Big).\nonumber\\
\end{align*}
This implies that the root number can be written as
\begin{displaymath}
W(E_t)=-W_2(E_t)W_3(E_t)\left(\frac{-2}{\vert t_{(6)}\vert}\right)\left(\frac{-1}{\mathrm{sq}(t)_{(6)}}\right).
\end{displaymath}

Its value depends on the remainder of $\tau_1$ modulo $2^63^4$ and $t_2$ modulo 4.

\end{proof}


\begin{rem}
In \cite{CS}, Cassels and Schinzel show that the elliptic surfaces given by the equations
\begin{equation}\label{CS1}
y^2=x(x^2-(1+t^4)^2)
\end{equation}
and
\begin{equation}\label{CS7}
y^2=x(x^2-49(1+t^4)^2)
\end{equation}
are such that every fibre at $t\in\mathbb{Q}$ have the same root number. On (\ref{CS1}), the root number is always $+1$ and on (\ref{CS7}), always $-1$.

Remark that it is possible to find a $N_1$ lower than the one given by Theorem \ref{thmB} such that the function $t\rightarrow W(E_{f(t)})$ is $N_1$-periodic. This value is $N_1=2^43^2=144$.

The only remainder values of $-(1+t^4)^2\mod16$ are $-1$ and $-4$, and the only remainder value modulo $9$ is $-1$. By the Chinese Remainder Theorem, we find a small set of possible remainders modulo $144$ for $-(1+t^4)^2$, and the associated root number on the congruence class they define is $+1$. 
It it is not possible to conclude anything on the density of the rational points from the root number of the fibre of surface (\ref{CS1}). 
On the surface (\ref{CS7}), the root number is $-1$ for every possible remainder of $-7^2(1+t^4)^2$. Assuming the parity conjecture, the rational points of the surface (\ref{CS7}) are dense.

In \cite{Zhizhong} Huang presents a geometric method that proves the density of rational points for elliptic surfaces defined by the equation $y^2=x^3-d^2(1+t^4)^2x$, for infinitely many squarefree values of $d$ including $d=1$ and $7$.
\end{rem}

\subsection*{Acknowledgment}

This article is based on results obtained by the author in her PhD thesis \cite{Desjardinsthese}. She is grateful to Marc Hindry for his excellent supervision. She thanks David Rohrlich for carefully reading a previous version of this work. She also wishes to thank Samuele Anni, Vladimir Dokchitser, Anthony V\'arilly-Alvarado and Adam Morgan for helpful conversations, and the anonymous referee for useful suggestions. The author was supported by the Max Planck Institut f\"ur Mathematik in Bonn.

\bibliographystyle{alpha}
\bibliography{bibliothese}

\end{otherlanguage}
\end{document}